\documentclass[12pt]{amsart}

\usepackage{amssymb}
\usepackage{amsfonts}
\usepackage{amsthm}
\usepackage{graphicx}
\usepackage{color}
\usepackage{url}
\usepackage{microtype}
\usepackage{tikz}

\usepackage[symbol]{footmisc}

\usepackage{hyperref}

\definecolor{darkblue}{RGB}{0,0,160}
\hypersetup{
colorlinks,%
citecolor=black,%
filecolor=black,%
linkcolor=darkblue,%
urlcolor=darkblue
}

\usepackage[utf8]{inputenc}
\usepackage{todonotes}

\newtheorem{thm}{Theorem}[section]
\newtheorem{lemma}[thm]{Lemma}

\newtheorem{prop}[thm]{Proposition}
\newtheorem{conj}[thm]{Conjecture}

\theoremstyle{definition}
\newtheorem{example}[thm]{Example}
\newtheorem{remark}[thm]{Remark}
\newtheorem{defn}[thm]{Definition}
\newtheorem{conv}[thm]{Convention}

\numberwithin{equation}{section}

\newcommand{\ring}[1]{\ensuremath{\mathbb{#1}}}

\renewcommand{\>}{\rangle}

\newcommand\NN{\ring{N}}

\newcommand\RR{\ring{R}}

\newcommand\ZZ{\ring{Z}}

\newcommand\kk{\Bbbk}

\newcommand\pp{{\mathfrak p}}

\newcommand\xx{{\mathbf x}}
\newcommand\yy{{\mathbf y}}

\newcommand\cC{{\mathcal C}}

\newcommand\cG{{\mathcal G}}

\newcommand\spol{\mathrm{spol}}

\newcommand\iso{\cong}

\newcommand\til{\mathord\sim}

\DeclareMathOperator\charac{char} 
\DeclareMathOperator\diag{diag} 
\DeclareMathOperator\sat{sat} 
\DeclareMathOperator\im{im_\ZZ} 

\newcommand{\inD}[1][\relax]{\def\argone{#1}\def\temprelax{\relax}
  \ifx\argone\temprelax\right.\else\,\middle|#1\right.{}\fi}

\newcommand{\ideal}[1]{\left\langle #1 \right\rangle}

\newcommand{\lattice}{\mathcal{L}_G}
\renewcommand{\sat}[1]{\mathcal{J}_{#1}}
\newcommand{\pbei}[1]{\mathcal{I}_{#1}}
\newcommand{\saturation}{\pbei{G}:(\prod_{i\in\vertices{G}}x_iy_i)^\infty}
\newcommand{\gb}{\cG_\succ(G)}

\newcommand{\vertices}[1]{V(#1)}
\newcommand{\movesOdd}{\mathcal{M}^{\textnormal{odd}}_G}
\newcommand{\movesEven}{\mathcal{M}^{\textnormal{even}}_G}

\newcommand{\bComp}{c_0}
\newcommand{\nbComp}{c_1}
\newcommand{\cComp}{c}
\newcommand{\sComp}{\mathfrak{s}}
\newcommand{\nonbipPos}[1]{\mathfrak{p}^+(#1)}
\newcommand{\nonbipNeg}[1]{\mathfrak{p}^-(#1)}
\renewcommand{\subset}{\subseteq}
\renewcommand{\supset}{\supseteq}

\renewcommand{\int}{\mathrm{int}}

\begin{document}

\title{Parity binomial edge ideals}

\author{Thomas Kahle}
\address{Otto-von-Guericke Universität Magdeburg\\ Magdeburg, Germany} 
\urladdr{\url{http://www.thomas-kahle.de}}

\author{Camilo Sarmiento}
\address{Otto-von-Guericke Universität Magdeburg\\ Magdeburg, Germany} 
\urladdr{\url{http://www.uni-magdeburg.de/sarmient/}}

\author{Tobias Windisch}
\address{Otto-von-Guericke Universität Magdeburg\\ Magdeburg, Germany} 
\urladdr{\url{http://www.uni-magdeburg.de/windisch/}}

\date{\today}

\makeatletter
  \@namedef{subjclassname@2010}{\textup{2010} Mathematics Subject Classification}
\makeatother




\begin{abstract}
Parity binomial edge ideals of simple undirected graphs are
introduced.  Unlike binomial edge ideals, they do not have square-free
Gröbner bases and are radical if only if the graph is bipartite or the
characteristic of the ground field is not two.  The minimal primes are
determined and shown to encode combinatorics of even and odd walks in
the graph.  A mesoprimary decomposition is determined and shown to be
a primary decomposition in characteristic two.
\end{abstract}

\maketitle

\setcounter{tocdepth}{1}
\tableofcontents

\section{Introduction}
A binomial is a polynomial with at most two terms and, a binomial ideal
is a polynomial ideal generated by binomials.  Binomial ideals appear
frequently in mathematics and also applications to statistics and
biology.  This paper is about decompositions of binomial ideals which
appear, for instance, in understanding the implications of conditional
independence statements
\cite[Chapter~3]{drton09:_lectur_algeb_statis}, steady states of
chemical reaction networks~\cite{MillanToricSteady,conradiKahle15}, or
combinatorial game theory~\cite{miller09:_theor,millerMisere13}.

Decomposition theory of binomial ideals started with Eisenbud and
Sturmfels' fundamental paper~\cite{es96}, which proves the existence
of binomial primary decomposition over algebraically closed fields.
It can be seen, however, that the field assumption is not strictly
necessary: a mesoprimary decomposition captures all combinatorial
features and exists over any given field~\cite{kahle11mesoprimary}.
Separating the arithmetical and combinatorial aspects of binomial
ideals is important for applications where binomial primary
decompositions over the complex numbers are often inadequate since
they obscure combinatorics and prevent interpretations of the
indeterminates as, say, probabilities or concentrations.

Actual primary decompositions have been computed almost exclusively of
radical ideals.  It is a general feature of (meso)primary
decomposition that the embedded primes and components remain elusive.
The partial decomposition of the Mayr-Meyer ideals by Swanson
illustrates quite beautifully the mess one typically encounters when
trying to determine components over embedded
primes~\cite{swanson04:_mayr_meyer_embedded}.  The minimal primes are
often combinatorially fixed and thus much better behaved.  For
instance, for lattice basis ideals they are entirely determined by the
indeterminates they contain~\cite{hosten00:_primar_lattice_basis}.
More examples of interesting combinatorial descriptions of minimal
primes of binomial ideals appear, for instance,
in~\cite{Herzog2010,hosten04:_adjacent_minors,kahle12:positive-margins}.
In practice, binomial (primary) decompositions can be found with
computer algebra.  For experimentation we used and recommend the
packages \textsc{Binomials}~\cite{kahle11:binom-jsag} and
\textsc{BinomialEdgeIdeals}~\cite{windisch15-bei} in
\textsc{Macaulay2}~\cite{M2}.

This paper is about a class of ideals whose primary decomposition
depends on the characteristic of the field and is in general different
from the mesoprimary decomposition.  We decompose these ideals using a
new technique and hope to add to the toolbox for binomial
decompositions.  To define the key player, let $G$ be a simple
undirected graph on $\vertices{G}$ and with edge set~$E(G)$.  Let
$\kk$ be any field and denote by
$\kk[\xx,\yy] = \kk[x_i,y_i : i\in\vertices{G}]$ the polynomial ring
in $2|\vertices{G}|$ indeterminates.
\begin{defn}\label{d:pbei}
The \emph{parity binomial edge ideal of $G$} is
\[
\pbei{G} := \ideal{x_ix_j - y_iy_j : \{i,j\} \in E(G)} \subset \kk[\xx,\yy].
\]
\end{defn}
Parity binomial edge ideals share a number of properties with binomial
edge ideals~\cite{Herzog2010}, but the combinatorics is subtler.
Various properties related to walks in~$G$ depend on whether the walk
has even or odd length (and hence the name).  If~$G$ is bipartite,
then everything reduces to the results of~\cite{Herzog2010} as
follows.

\begin{remark}\label{r:bipartite-bei}
Let $G$ be bipartite on the vertex set $V_1 \dot\cup V_2$.  Consider
the ring automorphism of $\kk[\xx,\yy]$ which
exchanges $x_i$ and $y_i$ if $i\in V_1$ and leaves all remaining
indeterminates invariant. Under this automorphism, $\pbei{G}$ is the
image of the binomial edge ideal of~$G$.
\end{remark}

Definition~\ref{d:pbei} was suggested by Rafael Villarreal at the
MOCCA Conference~2014 in Levico Terme.  He asked if parity binomial
edge ideals are radical. Theorem~\ref{t:radical}
combined with Remark~\ref{r:RadicalityOverFiniteField}
says that this is the case if and only if $G$ is bipartite, or
$\charac(\kk) \neq 2$.  We compute the minimal primes of $\pbei{G}$ in
Section~\ref{s:minprimes}.  In Proposition~\ref{p:Intersection}, we
write $\pbei{G}$ as an intersection of binomial ideals whose
combinatorics is simpler, since then a short induction shows that,
under the field assumption, all occurring intersections are radical
(Theorem~\ref{t:radical}) and hence $\pbei{G}$ is radical.  In
$\charac(\kk) = 2$ we determine a primary decomposition
(Theorem~\ref{t:charzwo}), which turns out to be also a mesoprimary
decomposition (Theorem~\ref{t:meso}).

Our determination of the minimal primes goes a route that is familiar
from~\cite{kahle12:positive-margins}.  We first determine generators
of the distinguished component $\saturation$ (that is, a Markov basis)
in Section~\ref{s:markov}.  Binomials $b$ that appear in the
Markov~basis but are not themselves contained in $\pbei{G}$ have the
property that $mb \in \pbei{G}$ for some monomial~$m$.  This means
that $\pbei{G}:b$ contains the monomial $m$ and thus some minimal
primes of $\pbei{G}$ contain the indeterminates that constitute~$m$.
In the case of parity binomial edge ideals, the witness monomial can
be found inductively using walks (Lemma~\ref{l:Walks}).

Just looking at Definition~\ref{d:pbei} one may hope that parity
binomial edge ideals~would deform to monomial edge ideals under the
Gröbner deformation. This is~not~the~case as already the simplest
examples show, but nevertheless, the lexicographic~Gröbner basis has
combinatorial structure and we describe it completely in
Section~\ref{s:GroebnerBasis}.

Shortly before first posting this paper on the arXiv, the authors
became aware of~\cite{herzog2014ideal}.  That paper contains a
different analysis of radicality of parity binomial edge ideals.  In
characteristic two, the parity binomial edge ideal $\pbei{G}$
coincides with the ideal~$L_G$ defined there; thus radicality is
clarified by their Theorem~1.2 which here appears as
Remark~\ref{r:RadicalityOverFiniteField}.  If the characteristic of
$\kk$ is not two, the linear transformation $x_i \mapsto x_i - y_i$,
$y_i \mapsto x_i + y_i$ maps the parity binomial edge ideal to the
permanental edge ideal $\Pi_G$ defined
in~\cite[Section~3]{herzog2014ideal}.  Radicality of this ideal is
clarified in their Corollary~3.3 by means of a Gröbner bases
calculation.  Our approach here is different and was developed
completely independently.  In particular, our proof of radicality
cannot use the Gröbner basis by Remark~\ref{r:nosqfreeGB}.
Additionally we can clarify the separation of combinatorics and
arithmetics of $\pbei{G}$ independent of $\charac(\kk)$ and determine
its mesoprimary decomposition.

\subsection*{Conventions and notation}
For $n\in\NN_{>0}$, let $[n] := \{1,\dots,n\}$.  All graphs here are
finite and simple, that is, they have no loops or multiple edges.  For
any graph~$G$, $\vertices{G}$ is the vertex set and $E(G)$ is the edge
set. For any $S\subset \vertices{G}$, $G[S]$ is the induced subgraph
on~$S$ and for a sequence of vertices
$P=(i_1,\ldots,i_r)\in\vertices{G}^r$, $G[P]:=G[\{i_1,\ldots,i_r\}]$.
Throughout we assume that $G$ is connected and in particular has no
isolated vertices if $|V(G)|\ge 2$.  According to
Definition~\ref{d:pbei}, if a graph is not connected then the parity
binomial edge ideals of the connected components live in polynomial
rings on disjoint sets of indeterminates such that the problem reduces
to connected graphs.  Despite this assumption, non-connected graphs
appear.  Thus, for any graph~$H$, let $\cComp(H)$ be the number
connected components, $\bComp(H)$ the number of bipartite connected
components, and $\nbComp(H)$ the number of connected components which
contain an odd cycle.  We freely identify ideals of sub-polynomial
rings of $\kk[\xx,\yy]$ with their images in~$\kk[\xx,\yy]$.
Likewise~ideals of $\kk[\xx,\yy]$ that do not use some of the
indeterminates are considered ideals of the~respective subrings.  A
binomial is \emph{pure difference} if it equals the difference of two
monomials.

\subsection*{Acknowledgements}
The authors would like to thank Rafael Villareal for posting the
question of radicality of parity binomial edge ideals.  We thank
Fatemeh Mohammadi for pointing us at~\cite{herzog2014ideal}.  The
authors appreciate the many comments and suggestions by Issac Burke
and Mourtadha Badiane.  T.K. and C.S. are supported by the Center for
Dynamical Systems (CDS) at Otto-von-Guericke University Magdeburg.
T.W.  is supported by the \href{https://www.studienstiftung.de}{German
National Academic Foundation} and
\href{https://www.ma.tum.de/TopMath/WebHomeEn}{TopMath}, a graduate
program of the \href{https://www.elitenetzwerk.bayern.de }{Elite
Network of Bavaria}.

\section{Markov bases}\label{s:markov}
Markov bases were first defined for toric ideals, but the definition
extends easily to other lattice ideals.  In this paper, by a Markov
basis we mean generators of $\saturation$, which is compatible with
the extended notions of Markov bases used in
\cite[Section~1.3]{drton09:_lectur_algeb_statis} and
\cite[Section~2.1]{rauh2014lifting}.

\begin{defn}
Let $G$ be a graph.  A $(v,w)$-\emph{walk} of \emph{length $r-1$} is a
sequence of vertices $v = i_1, i_2 \dots ,i_r = w$ such that
$\{i_k,i_{k+1}\}\in E(G)$ for all $k\in[r-1]$.  The walk is \emph{odd}
(\emph{even}) if
its length is odd (even).  A \emph{path} is walk that uses no vertex
twice.  A \emph{cycle} is a walk with $v=w$.  The \emph{interior} of a
$(v,w)$-walk $P=(i_1,\dots,i_r)$ is the set
$\int(P)=\{i_1,\dots,i_r\}\setminus\{v,w\}$.
\end{defn}

\begin{remark}
In this paper, a cycle is only defined with a marked start and end
vertex.  Consequently the interior of a cycle (in the usual graph
theoretic sense) also depends on the choice of this vertex.
\end{remark}

\begin{conv}\label{r:walkTalk}
When no ambiguity can arise, for instance because the vertices are
explicitly enumerated, we call a $(v,w)$-walk simply a walk. 
\end{conv}

\begin{figure}[bt]
	\begin{tikzpicture}[xscale=0.5,yscale=0.5]
		\node [label={[label distance=1pt]270:$4$},fill, circle, inner sep=2pt](a) at (0,0) {};
		\node [label={[label distance=1pt]90:$1$},fill, circle, inner sep=2pt](b) at (1,2) {};
		\node [label={[label distance=1pt]270:$3$},fill, circle, inner sep=2pt](c) at (2,0) {};
		\node [label={[label distance=1pt]270:$6$},fill, circle, inner sep=2pt](d) at (4,0) {};
		\node [label={[label distance=1pt]270:$5$},fill, circle, inner sep=2pt](e) at (6,0) {};
		\node [label={[label distance=1pt]90:$2$},fill, circle, inner sep=2pt](f) at (3,2){};
		\draw(a)--(c) --(d) --(e);
		\draw(c) -- (b) -- (f) -- (c);
	\end{tikzpicture}
	\caption{\label{fig:RunningExample} A graph with an even walk,
	but no even path from~$4$ to~$5$.  The interior of the walk
	$(4,3,1,2,3,6,5)$ is~$\{1,2,3,6\}$.  }
\end{figure}
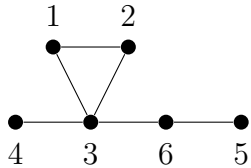

\begin{lemma}\label{l:Walks}
Let $P$ be an $(i,j)$-walk in $G$ and for $k\in\int(P)$, let
$t_k\in\{x_k,y_k\}$ arbitrary. If $P$ is odd, then
\begin{equation*}\label{equ:OddWalks}
(x_ix_j-y_iy_j)\prod_{k\in\int(P)}t_k\in\pbei{G}.
\end{equation*}
If $P$ is even, then
\begin{equation*}\label{equ:EvenWalks}
(x_iy_j-y_ix_j)\prod_{k\in\int(P)}t_k\in\pbei{G}.
\end{equation*}
\end{lemma}
\begin{proof}
We prove the statement by induction on the length $r$ of $P$. If
$r=1$, the statement is true by definition, thus assume that $r>1$. If
$\int(P)=\emptyset$, then $P$ is odd, $i$ is adjacent to $j$, and the
claim holds trivially.  If $\int(P)\neq\emptyset$, pick a vertex
$s\in\int(P)$.  Consider first the case that $P$ is an odd walk.
Exchanging the roles of~$i$ and $j$ if necessary, we can assume that
the $(i,s)$-subwalk of $P$ is odd and that the $(s,j)$-subwalk is
even.  Using the induction hypothesis, the binomials corresponding to
these walks are in $\pbei{G}$. Now, if $t_s=x_s$, then
\begin{equation*}
x_ix_sx_j\prod_{k\in\int(P)\setminus s} t_k
\equiv_{\pbei{G}}y_iy_sx_j \prod_{k\in\int(P)\setminus s} t_k
\equiv_{\pbei{G}}y_ix_sy_j \prod_{k\in\int(P)\setminus s} t_k
\end{equation*}
where we have first applied a binomial corresponding to the odd
$(i,s)$-subwalk (which may traverse~$j$) and then a binomial
corresponding to the even $(s,j)$-subwalk of $P$ (which may
traverse~$i$). If $t_s=y_s$, then we first apply the $(s,j)$-walk and
then the $(i,s)$-walk. The induction step for an even walk is similar
and omitted.
\end{proof}

\begin{remark}\label{r:oddCycle}
Lemma~\ref{l:Walks} also holds for odd cycles in which case we get
that monomial multiples of $x_i^2-y_i^2$ are contained in $\pbei{G}$
for any vertex $i$ that is contained in the same connected component
as an odd cycle.
\end{remark}

Let $\{i,j\}\in E(G)$ and denote
$m_{\{i,j\}}:=e_i+e_j\in\ZZ^{\vertices{G}}$ where $e_i$ is the
standard unit vector in~$\ZZ^{\vertices{G}}$ corresponding to
$i\in\vertices{G}$.  With this notation, the generator $x_ix_j-y_iy_j$
has exponent vector
$(m_{\{i,j\}},-m_{\{i,j\}})^T \in \ZZ^{2|\vertices{G}|}$.  The
exponent vectors of generators of $\pbei{G}$ generate a lattice
\begin{equation*}
\lattice= \ZZ \left\{ \begin{pmatrix}m_e\\-m_e\end{pmatrix} : e \in
E(G) \right\}= \im \begin{pmatrix}A_G\\-A_G\end{pmatrix} \subset
\ZZ^{2|\vertices{G}|},
\end{equation*}
where $A_G$ is the incidence matrix of~$G$.  Consequently $\lattice$
is the Lawrence lifting of~$\im (A_G) \subset \ZZ^n$.  Recall that a
Graver basis of a lattice is the unique minimal subset of the lattice
such that each element of the lattice is a sign-consistent linear
combination of elements of the Graver basis (see
\cite[Chapter~3]{Loera2013} for Graver basics).  A standard fact about
Lawrence liftings is that the Graver basis of $\im (A_G)$ can be
lifted to a Graver basis of $\lattice$, which here equals the
universal Gröbner basis and any minimal Markov basis of~$\lattice$
\cite[Proposition~1.1]{Bayer1999Lawrence}.  To determine the Graver
basis of~$\im(A_G)$, let
\begin{equation*}
	\begin{split}
	\movesOdd  & :=\{e_i+e_j: \text{ there is an odd $(i,j)$-walk in $G$}\}\\
	\movesEven & :=\{e_i-e_j: \text{ there is an even $(i,j)$-walk in $G$}\} \setminus \{0\}.
	\end{split}
\end{equation*}
Note in particular that if there is an odd $(i,i)$-walk, then
$2\cdot e_i\in\movesOdd$.

\begin{prop}\label{p:graver}
The Graver basis of $\im(A_G)$ is $\pm(\movesOdd\cup\movesEven)$.
\end{prop}
\begin{proof}
According to Pottier's termination criterion \cite[Algorithm
3.3]{Loera2013}, it suffices to check that the sum of two elements of
$\pm(\movesOdd\cup\movesEven)$ can be reduced to zero
sign-consistently.  
If there are no cancelations in the sum, for example if the two
summands have disjoint support, the sum is reduced by either of the
summands.  Cancelation among elements $e_{i_1} \pm e_{i_2}$ and
$e_{j_1} \pm e_{j_2}$ can only occur if
$|\{i_1,i_2,j_1,j_2\}| \leq 3$.  Without loss of generality assume
$i_2 = j_1$.  Thus, if cancelation occurs, the sum of two proposed
Graver elements must equal $\pm(e_{i_1} \pm e_{j_2})$ and this is
either zero or another element in $\pm (\movesOdd\cup\movesEven)$ by
concatenation of walks.
\end{proof}

Proposition~\ref{p:graver} shows that the minimal Markov, or
equivalently Graver, basis of the ideal saturation
$\sat{G}:=\saturation$ at the coordinate hyperplanes consists of the
following binomials:

\begin{prop}\label{p:markov}
\begin{equation}\label{eq:markovBasis}
\begin{split}
\sat{G} = & \ideal{x_ix_j-y_iy_j: \text{ there is an odd $(i,j)$-walk
in $G$}} \\
+ & \ideal{x_iy_j-y_jx_i: \text{ there is an even $(i,j)$-walk in $G$} }.
\end{split}
\end{equation}
\end{prop}
\begin{proof}
This is Proposition \ref{p:graver} and
\cite[Proposition~1.1]{Bayer1999Lawrence}.
\end{proof}

\begin{example}
Due to the odd cycle in the graph $G$ in
Figure~\ref{fig:RunningExample}, for all pairs $(i,j)$ of vertices
with~$i\neq j$, both~$x_ix_j - y_iy_j$ and $x_iy_j - x_jy_i$ are
contained in~$\sat{G}$.  Hence, the ideal $\sat{G}$ has $15$
generators for odd walks and $15$ for even walks with disjoint
endpoints. Since $G$ is not bipartite, $x_i^2-y_i^2\in\sat{G}$ for all
$i\in[6]$. In total, a minimal Markov basis of $\sat{G}$ consists of
$36$ generators.
\end{example}

\begin{remark}
If $G$ is bipartite, the reachability of vertices with even or odd
walks is determined by membership in the two groups of vertices.
Consequently, for each spanning tree $T\subset G$ we have
$\sat{T}=\sat{G}$.  This is not true if $G$ has an odd cycle.
\end{remark}

\section{A lexicographic Gröbner basis}\label{s:GroebnerBasis}

For this section, an ordering of $\vertices{G}$ is necessary.  Fix any
labeling $\vertices{G}\iso [n]$ and let $\succ$ be the lexicographic
ordering on $\kk[\xx,\yy]$ induced by $x_{1}\succ\cdots\succ
x_{n}\succ y_{1}\succ\cdots\succ y_{n}$.  For $i,j\in\vertices{G}$
write $i\succ j$ if $x_{i}\succ x_{j}$. 
We now describe the lexicographic Gröbner basis of the parity binomial
edge ideal. This Gröbner basis reduces the binomial of an $(i,j)$-walk
$P$ by the binomial of an $(i,k)$-walk or that of a $(k,j)$-walk for a
suitable $k\in\int(P)$. For example, if $P$ is odd and if $i\succ
j\succ k$, then a binomial of $P$ can be reduced to zero by the
binomials corresponding to two subwalks $(i,k)$ and $(k,j)$ in $G[P]$,
independently of their parity and of the chosen binomial of~$P$.
The following definition identifies the configurations that lead to
irreducible walk~binomials. 

\begin{defn}\label{d:ReducedWalks}
Let $P$ be an odd
$(i,j)$-walk with $i\succeq j$ and for $k\in\int(P)$, let
$t_k\in\{x_k,y_k\}$ be arbitrary.  The binomial
\begin{equation}\label{eq:oddReducedWalk}
(x_ix_j-y_iy_j)\prod_{k\in\int(P)}t_k
\end{equation}
is \emph{reduced} if for all $k\in\int(P)$
\begin{itemize}
\item there is no odd $(i,j)$-walk in $G[P\setminus\{k\}]$,
\item $k\succ j$,
\item if $i\succ k \succ j$, then all $(i,k)$-walks in $G[P]$ are 
odd and $t_k=y_k$, and
\item if $k\succ i\succ j$, then $t_k=y_k$.
\end{itemize}

Let $P$ be an even $(i,j)$-walk with $i\succ j$, and for
$k\in\int(P)$, let $t_k\in\{x_k,y_k\}$ be arbitrary.  The binomial
\begin{equation}\label{eq:evenReducedWalk}
(x_iy_j-y_ix_j)\prod_{k\in\int(P)}t_k
\end{equation}
is \emph{reduced} if for all $k\in\int(P)$
\begin{itemize}
\item there is no even $(i,j)$-walk in $G[P\setminus\{k\}]$,
\item if $i\succ j\succ k$, then all $(i,k)$-walks in $G[P]$ are
either odd and $t_k=y_k$ or they are all even and $t_k=x_k$,
\item if $i\succ k \succ j$, then all $(i,k)$-walks in $G[P]$ are odd
and $t_k=y_k$, and
\item if $k\succ i\succ j$, then $t_k=y_k$.
\end{itemize}
The set of reduced binomials is written $\gb$. 
\end{defn}

Clearly, $x_ix_j-y_iy_j\in\gb$ for every edge $\{i,j\}\in E(G)$. We
make the reduced binomials more explicit as follows.

\begin{remark}\label{r:writeExplicit}
Let $i\succeq j$ and let $P$ be an $(i,j)$-walk in $G$ with
$\int(P)=\{i_1,\dots,i_r\}$.  Assume that there exists variables
$t_{k}\in\{x_{k},y_{k}\}$, $k\in\int(P)$, such that the respective
binomial in equation \eqref{eq:oddReducedWalk} or
\eqref{eq:evenReducedWalk} is reduced. The case distinction in
Definition~\ref{d:ReducedWalks} fixes the value of $t_k$ for any
$k\in\int(P)$ as follows. Let
\begin{equation*}
P^x:=\{k\in\int(P): j\succ k\text{ and there is an even $(i,k)$-walk
in $G[P]$}\}
\end{equation*}
and $P^y:=\int(P)\setminus P^x$.  In particular, $P^y=\int(P)$ if $P$
is odd.  Thus,
\begin{equation*}
\prod_{k\in\int(P)}t_k=\prod_{k\in P^x}x_k\prod_{k\in P^y}y_k.
\end{equation*}
\end{remark}

\begin{example}\label{e:evenReducedWalks}
Let $P$ be an even $(i,j)$-walk with $i\succ j$ such that there exists
$k\in\int(P)$ with $j\succ k$ and such that there exists an even and
an odd $(i,k)$-walk in $G[P]$.  If $t_k=x_k$ in equation
\eqref{eq:evenReducedWalk}, then the binomial can be reduced by the
binomial corresponding to the odd $(i,k)$-walk.  If $t_k=y_k$, then
the binomial can be reduced to zero by the binomial corresponding to
the even $(i,k)$-walk.
\end{example}

\begin{example}\label{ex:ReducedOddWalksCanHaveCycles}
Consider the even walk $(4,3,1,2,3,6)$ in the parity
binomial edge ideal for Figure~\ref{fig:RunningExample}. The
binomial~$(x_4x_6-y_4y_6)y_3y_2y_1$ is reduced, whereas the binomial
$(x_4x_6-y_4y_6)y_3x_2y_1$ is not.  In particular, reduced odd walks
can have odd cycles. In the even $(4,5)$-walk $P=(4,3,1,2,3,6,5)$
there exists an even $(4,6)$-subwalk and an odd $(4,6)$-subwalk in
$G[P]$. Since $4\succ 6$, no choice of variables $t_k\in\{x_k,y_k\}$
makes the binomial $(x_4y_5-y_4x_5)t_1t_2t_3t_4t_6$ a reduced
binomial.
\end{example}

The first step is to see that reduced binomials have minimal leading
terms among all binomials in Lemma~\ref{l:Walks} corresponding to
walks, justifying their name.
\begin{lemma}\label{l:ReductionOfWalkBinomials}
Let $P$ be an $(i,j)$-walk and $t_k\in\{x_k,y_k\}$ for $k\in\int(P)$
arbitrary. Then, $(x_ix_j-y_iy_j)\prod_{k\in\int(P)}t_k$ if $P$ is
odd, and $(x_iy_j-y_ix_j)\prod_{k\in\int(P)}t_k$ if $P$ is even,
reduce to zero modulo~$\gb$.
\end{lemma}
\begin{proof}
This is an induction on the length of $P$. Assume that the
binomial contradicts the first bullet in the respective definition of
being reduced, then it is a monomial multiple of a binomial of a
shorter walk which can be reduced to zero by induction.  If the
binomial fulfills the first bullet, then there exists some
$k\in\int(P)$ that violates one of the other properties in the
definition.  In this case, there exists two subwalks $(i,k)$ and
$(k,j)$ whose binomials reduce the original binomial (see
Example~\ref{e:evenReducedWalks}), and which are
themselves reducible by the induction hypothesis.
\end{proof}

We now state the main theorem of this section.  Its proof is by
Buchberger's criterion and splits into a couple of lemmas.
\begin{thm}\label{thm:GB}
The set $\gb$ of reduced binomials is the reduced Gröbner basis of
$\pbei{G}$ with respect to~$\succ$.
\end{thm}

\begin{remark}\label{r:Herzog}
The Gröbner basis in \cite{Herzog2010} looks similar, but for the
original binomial edge ideals there are no binomials corresponding to
$k\in\int(P)$ with $i \succ k \succ j$ in the Gröbner basis. The
Gröbner basis there is also not a subset of our Gröbner basis.  For
example, all Gröbner elements in \cite{Herzog2010} which come from
\emph{admissible} $(i,j)$-paths can be reduced to zero by odd moves
from $\gb$ if there exists an odd $(i,k)$-subwalk with $j\succ k$.
\end{remark}

For the reduction of s-polynomials we use the following well-known
fact.
\begin{lemma}\label{l:RegularSequences}
Let $f,g\in\kk[\xx]$ and $\succ$ a monomial ordering. If their leading
monomials form a regular sequence, then $\spol(uf,vg)$ reduces to zero
for all monomials $u,v\in\kk[\xx]$\footnote[1]{It is important to note
that the reduction in this lemma is by $f$ and~$g$, not $uf$ and $vg$.
In the proofs of Lemmas~\ref{l:GBEvenWalks},~\ref{l:GBOddWalks},
and~\ref{l:GBMixedWalks} it is applied wrongly.  The elements $f$ and
$g$ there are not contained in the parity binomial edge ideal and the
lemma cannot be used to deduce that these S-pairs reduce to zero.
However, all cases in which we have wrongly applied
Lemma~\ref{l:RegularSequences} are straight-forward and
Theorem~\ref{thm:GB} is true. Moreover, the universal Gröbner basis of
parity binomial edge ideals is computed in
\cite[Chapter~6]{windisch2017}, from which the reduced Gröbner basis
with respect to $\succ$ as stated in Theorem~\ref{thm:GB} can be
deduced.  This error exists in the published version and this arXiv
version agrees with the published version with the exception of this
footnote.  We thank Aldo Conca for pointing out this problem.}.
\end{lemma}

\begin{lemma}\label{l:GBEvenWalks}
Let $g_P$ and~$g_Q$ be reduced binomials corresponding to even walks
$P$ and~$Q$.  Then $\spol(g_P,g_Q)$ reduces to zero with respect
to~$\gb$.
\end{lemma}
\begin{proof}
Let $P$ be an even $(p_1,p_2)$-walk with $p_1\succ p_2$ and $Q$ be an
even $(q_1,q_2)$-walk with $q_1\succ q_2$.  By
Remark~\ref{r:writeExplicit} we write
\begin{equation*}
\begin{split}
g_P=(x_{p_1}y_{p_2}-y_{p_1}x_{p_2})\cdot\prod_{i\in
P^x}x_i\cdot\prod_{i\in P^y}y_i\\
g_Q=(x_{q_1}y_{q_2}-y_{q_1}x_{q_2})\cdot\prod_{i\in
Q^x}x_i\cdot\prod_{i\in Q^y}y_i.
\end{split}
\end{equation*}
If $|\{p_1,p_2,q_1,q_2\}| = 4$, then $x_{p_1}y_{p_2}$ and
$x_{q_1}y_{q_2}$ are coprime and thus form a regular sequence.  Lemma
\ref{l:RegularSequences} gives this case.  If
$\{p_1,p_2,q_1,q_2\} = \{q_1,q_2\}$, then the s-polynomial is zero.

The only interesting case is when $P$ and $Q$ have precisely one
endpoint in common.  First, let that common endpoint be $v:=p_1=q_1$.
Since $p_1\not\in Q^x$, $q_1\not\in P^x$, and since
we can assume that $q_2\succ p_2$, the s-polynomial is
\begin{equation*}
(x_{q_2}y_{p_2}-y_{q_2}x_{p_2})\cdot y_v\cdot \prod_{i\in P^x\cup
Q^x}x_i\prod_{i\in(P^y\cup Q^y)\setminus\{q_2,p_2\}}y_i.
\end{equation*}
This binomial is a monomial multiple of the binomial obtained from the
$(q_2,p_2)$-walk which might traverse the vertex $v=p_1=q_2$.  Hence,
the s-polynomial reduces to zero by
Lemma~\ref{l:ReductionOfWalkBinomials}.  The case that $p_2=q_2$ is
similar and omitted.  The last case is (without loss of generality)
$q_1\succ q_2=p_1\succ p_2$. In this case, $x_{q_1}y_{q_2}$ and
$x_{p_1}y_{p_2}$ form a regular sequence and due to
Lemma~\ref{l:RegularSequences} $\spol(g_P,g_Q)$ reduces to zero.
\end{proof}

\begin{lemma}\label{l:GBOddWalks}
Let $g_P$ and~$g_Q$ be reduced binomials corresponding to odd walks
$P$ and~$Q$.  Then $\spol(g_P,g_Q)$ reduces to zero with respect
to~$\gb$.
\end{lemma}
\begin{proof}
Assume that $P$ is a $(p_1,p_2)$-walk with $p_1\succ p_2$ and $Q$ is a
$(q_1,q_2)$-walk with $q_1\succ q_2$.  Without loss of generality, let
$p_1\succ q_1$.  By Lemma~\ref{l:RegularSequences}, we can assume
$|\{p_1,p_2\}\cap\{q_1,q_2\}|\ge 1$.  Clearly, if
$\{p_1,p_2\}=\{q_1,q_2\}$, $\spol(g_P,g_Q)=0$.  In total assume that
$\{p_1,p_2\}\neq\{q_1,q_2\}$.  Under this assumptions, in all
remaining cases, the s-polynomial is a monomial multiple of the
binomial corresponding to the even walk which arises from gluing $P$
and $Q$ along the vertex they have in common.
\end{proof}

\begin{lemma}\label{l:GBMixedWalks}
Let $g_P$ and~$g_Q$ be reduced binomials corresponding,
respectively, to an odd walk $P$ and an even walk~$Q$.  Then $\spol(g_P,g_Q)$
reduces to zero with respect to~$\gb$.
\end{lemma}
\begin{proof}
Let $P$ be an $(p_1,p_2)$-walk with $p_1\succeq p_2$ and $Q$ be an
even $(q_1,q_2)$-walk with $q_1\succ q_2$. By
Remark~\ref{r:writeExplicit} we write
\begin{equation*}
\begin{split}
g_P&=(x_{p_1}x_{p_2}-y_{p_1}y_{p_2})\prod_{i\in\int(P)}y_i,\\
g_Q&=(x_{q_1}y_{q_2}-y_{q_1}x_{q_2})\prod_{i\in Q^x}x_i\prod_{i\in Q^y}y_i.
\end{split}
\end{equation*}
By Lemma \ref{l:RegularSequences} it suffices to consider the case
that $q_1\in\{p_1,p_2\}$. If $p_1=q_1$, then
\begin{equation*}
\spol (g_P,g_Q) = (x_{p_2}x_{q_2}-y_{q_2}y_{p_2})y_{p_1}\prod_{i\in\int(P)\setminus
q_2}y_i\prod_{i\in Q^y}y_i\prod_{i\in Q^x\setminus p_2}x_i.
\end{equation*}
This s-polynomial is a monomial multiple of the binomial corresponding
to some $(p_2,q_2)$-walk, traversing $p_1=q_1$ if necessary.  Thus it
reduces by Lemma~\ref{l:ReductionOfWalkBinomials}.  The case that
$p_2=q_1$ is similar and omitted.
\end{proof}

\begin{proof}[Proof of Theorem~\ref{thm:GB}]
According to Lemma~\ref{l:GBEvenWalks}, Lemma~\ref{l:GBOddWalks}, and
Lemma~\ref{l:GBMixedWalks} the set $\gb$ fulfills Buchberger's
criterion and hence is a Gröbner basis of~$\pbei{G}$.  By
construction, the elements of $\gb$ are reduced with respect
to~$\succ$.
\end{proof}

Theorem~\ref{thm:GB} implies in particular that parity binomial edge
ideals of bipartite graphs are radical (which they must be by
Remark~\ref{r:bipartite-bei}).  This, however, does not require the
square-free initial ideal: if $\charac(\kk)\neq 2$ then all parity
binomial edge ideals are radical by Theorem~\ref{t:radical}.

\begin{remark}\label{r:nosqfreeGB}
The parity binomial edge ideal $\pbei{K_3}$ of the 3-cycle $K_3$,
cannot have a square-free initial ideal with respect to any monomial
order.  This follows from the fact that $\pbei{K_3}$ is not radical in
$\mathbb{F}_2[\xx,\yy]$ (see
Remark~\ref{r:RadicalityOverFiniteField}).  If
$\pbei{K_3}$ had a square-free Gröbner basis over some field $\kk$,
its binomials must be pure difference (since the generators of
$\pbei{K_3}$ are pure difference).  The pure difference property
yields that this Gröbner basis would also be a square-free Gröbner
basis over every other field, in particular, over~$\mathbb{F}_2$.
\end{remark}

\section{Minimal primes}\label{s:minprimes}

Generally, the minimal primes of a binomial ideal come in groups
corresponding to the sets of indeterminates they contain.  To start, we
determine the minimal primes of $\pbei{G}$ that contain no indeterminates,
that is, the minimal primes of~$\sat{G}$.  They follow quickly from
the next lemma, together with the results in \cite[Section~2]{es96}.

\begin{lemma}\label{l:smith}
Apart from zero rows, the Smith normal form of $\begin{pmatrix} A_G\\
-A_G \end{pmatrix}$
is the diagonal matrix $\diag(1,\dots,1,2,\dots,2)$ whose number of
entries $1$ is $|\vertices{G}|-\cComp(G)$ and the number of entries
$2$ equals~$\nbComp(G)$.
\end{lemma}
\begin{proof}
See \cite[Theorem~3.3]{Grossman1995}.
\end{proof}

The following ideals are the building blocks for the primary
decomposition of~$\sat{G}$.  For any connected graph $G$ with an odd
cycle, let
\[ \nonbipPos{G} = \ideal{x_i+y_i: i\in\vertices{G}} \quad \text{and}
\quad \nonbipNeg{G} := \ideal{x_i-y_i: i\in\vertices{G}}.
\]

\begin{prop}\label{p:decompSat}
Let $G$ be a graph consisting of bipartite connected components
$B_1,\dots,B_{\bComp(G)}$ and non-bipartite connected components
$N_1,\dots,N_{\nbComp(G)}$. If $\charac(\kk) \neq 2$, then $\sat{G}$
is radical, and its minimal primes are the $2^{\nbComp(G)}$ ideals
\[
\sum_{i=1}^{\bComp(G)}
\sat{B_i} + \sum_{i=1}^{\nbComp(G)} \pp^{\sigma_i}(N_i),
\]
where $\sigma$ ranges over $\{+,-\}^{\nbComp(G)}$.  On the other hand,
if $\charac(\kk) = 2$, then
\[
\sat{G} = \sum_{i=1}^{\bComp(G)} \sat{B_i} + \sum_{i=1}^{\nbComp(G)}
\sat{N_i}
\]
is primary of multiplicity $2^{\nbComp(G)}$ over the minimal prime
$\sum_{i=1}^{\bComp(G)} \sat{B_i} + \sum_{i=1}^{\nbComp(G)}
\pp^+(N_i)$.
\end{prop}
\begin{proof}
Assume first that $\kk$ is algebraically closed.  According to
\cite[Corollary~2.2]{es96}, the primary decomposition $\sat{G}$ is
determined by the saturations of the character that defines the
lattice ideal~$\sat{G}$.  If a graph is disconnected, then its
adjacency matrix has block structure according to the connected
components.  Therefore it suffices to assume that $G$ is connected.
If $G$ is bipartite, then Lemma~\ref{l:smith} and
\cite[Corollary~2.2]{es96} imply that the lattice ideal $\sat{G}$ is
prime.  We are thus left with the case that $G$ is connected and not
bipartite.

Assume first that $\charac(\kk) \neq 2$.  Lemma~\ref{l:smith}
and \cite[Corollary~2.2]{es96} together show that $\sat{G}$ is radical and
has two minimal primes.  We show that these are $\nonbipPos{G}$
and~$\nonbipNeg{G}$.  The first step is
$\sat{G}\subseteq\nonbipPos{G}$ using Proposition~\ref{p:markov}.  Let
$i, j\in\vertices{G}$, then
$x_ix_j-y_iy_j=x_i\cdot(x_j+y_j)-y_j\cdot(x_i+y_i)\in\nonbipPos{G}$
and
$x_iy_j-x_jy_i=x_i\cdot(x_j+y_j)- x_j\cdot(x_i+y_i)\in\nonbipPos{G}$.
Similarly, $\sat{G}\subseteq\nonbipNeg{G}$.
Now let $\pp\supset \sat{G}$ be a prime ideal.  If $\pp$ contains
$x_i + y_i$ for all $i$, then it is either equal to $\nonbipPos{G}$ or
not minimal over~$\sat{G}$.  If there exists a vertex $i$ such that
$x_i + y_i \notin \pp$, then since $G$ has an odd cycle and is
connected, for any vertex $j$ there are both an odd and an even
$(i,j)$-walk.  Thus,
\begin{equation*}
(x_i+y_i)\cdot (x_j-y_j) = x_ix_j  - y_iy_j + x_jy_i -  x_iy_j \in\pp.
\end{equation*}
Since $\pp$ is prime, it contains $x_j - y_j$ for each $j$ and thus
$\nonbipNeg{G} \subset\pp$.  This shows that $\nonbipNeg{G}$ and
$\nonbipPos{G}$ are the minimal primes of~$\sat{G}$.  

If $\charac(\kk) = 2$, then \cite[Corollary~2.2]{es96} gives that
$\sat{G}$ is primary of multiplicity two over a minimal prime which
equals~$\nonbipPos{G} = \nonbipNeg{G}$ by the above computation.
It is now evident that the algebraic closure assumption on $\kk$ is
irrelevant since all saturations of characters are defined over~$\kk$.
\end{proof}

\begin{remark}\label{r:charRadical}
The graph $G$ is bipartite if and only if $\sat{G}$ is prime.
\end{remark}

When decomposing a pure difference binomial ideal, all components
except those over the saturation $\sat{G}$ contain monomials (for a
combinatorial reason see \cite[Example~4.14]{kahle11mesoprimary}).
Our next step is to determine the indeterminates in the minimal primes.  To
this end, for any $S\subset\vertices{G}$ let $G_S$ be the induced
subgraph of $G$ on~$\vertices{G}\setminus S$ and
$\mathfrak{m}_S:=\ideal{x_s,y_s: s\in S}$.

\begin{lemma}\label{l:FormOfMinimalPrimes}
Let $\pp$ be a minimal prime of $\pbei{G}$. Then there exists
$S\subseteq\vertices{G}$ and a minimal prime $\pp'$ of $\sat{G_S}$
such that $\pp=\mathfrak{m}_S+\pp'$.
\end{lemma}
\begin{proof}
Let $S:=\{s\in\vertices{G}: x_s,y_s\in\pp\}$.  We first show the
inclusions
\begin{equation*}
\pbei{G}\subseteq \mathfrak{m}_S + \sat{G_S} \subseteq \pp.
\end{equation*}
The first inclusion is clear, while for the second, it suffices to
check that $\sat{G_S} \subseteq \pp$. Generators of $\sat{G_S}$
correspond to $(i,j)$-walks in $G_S$ according to
Proposition~\ref{p:markov}. Let $b$ be the binomial corresponding to
any such walk, and let
$\{k_1,\dots,k_r\} \subseteq\vertices{G}\setminus S$ be its
interior. By Lemma~\ref{l:Walks},
$t_{k_1}\cdots t_{k_r}\cdot b\in \pbei{G}\subseteq \pp$ for any choice
of indeterminates $t_{k_l} \in \{x_{k_l},y_{k_l}\}$, with~
$1\leq l \leq r$.  By the construction of $S$, there exists some
choice such that $t_{k_1}\cdots t_{k_r} \notin \pp$. Since $\pp$ is
prime, $b\in \pp$.  The minimal primes of
$\mathfrak{m}_S + \sat{G_S}$ arise as sums of
$\mathfrak{m}_S$ and minimal primes of~$\sat{G_S}$.  By
minimality,~$\pp$ equals $\mathfrak{m}_S+\pp'$ for some
minimal prime $\pp'$ of $\sat{G}$.
\end{proof}

Not all primes of the form $\mathfrak{m}_S+\pp'$ are minimal
over~$\pbei{G}$ (see
Example~\ref{ex:ChoicesOfBinomialPartOfSaturation}).  As for binomial
edge ideals, cut points play a crucial role in determining the sets
$S$ which lead to minimal primes, but for parity binomial edge ideals
we count connected components differently.  The bipartite ones count
double.

\begin{defn}\label{d:Disconnectors}
Let $\sComp(G)=\bComp(G)+\cComp(G) = 2\bComp(G) + \nbComp(G)$.  A set
$S\subset\vertices{G}$ is a \emph{disconnector} of~$G$ if
$\sComp(G_S)>\sComp(G_{S\setminus {\{s\}}})$ for every $s\in S$.
\end{defn}

\begin{remark}\label{r:EmptySetIsDisconnector}
The empty set is a disconnector of any graph, and disconnectors cannot
contain isolated vertices.
\end{remark}

\begin{remark}\label{r:DisconnectorsAndWalkParity}
If a graph $G$ has no isolated vertices, then
$\sComp(G_{\{s\}})\ge\sComp(G)$ for all $s\in\vertices{G}$ and
according to Definition~\ref{d:Disconnectors} a vertex $s$ is a
disconnector of~$G$ exactly if the inequality is strict.  Moreover,
one can conclude from the following proposition that $s$ is a
disconnector of $G$ if and only if
$\sat{G}\not\subset\mathfrak{m}_{\{s\}}+\sat{G_{\{s\}}}$.
\end{remark}

\begin{prop}\label{p:SubgraphsAndSaturation}
Let $G$ be a graph and $S\subseteq\vertices{G}$. Then
$\sat{G}\subset\mathfrak{m}_S+\sat{G_S}$ if and only if for
all $(i,j)$-walks in $G$ with $i,j\in\vertices{G_S}$, there is an
$(i,j)$-walk in $G_S$ of the same parity.
\end{prop}
\begin{proof}
Let $\sat{G}\subset\mathfrak{m}_S+\sat{G_S}$. Let
$m\in\sat{G}$ be a Graver move corresponding to an $(i,j)$-walk in $G$
with $i,j\not\in S$. Since
$m \in \kk[x_i,x_j,y_i,y_j]$, and no polynomial in $\sat{G_S}$ uses
indeterminates from $S$, we find $m \in \sat{G_S}$.  It follows that $m$ is
an element~of the Graver basis of $\sat{G_S}$ and thus corresponds to
an $(i,j)$-walk in~$G_S$ of the same~parity.

On the other hand, let $m\in\sat{G}$ be a move corresponding to a
$(i,j)$-walk in~$G$.  If $i\in S$ or $j\in S$, then
$m\in\mathfrak{m}_S$. If otherwise $i,j\in\vertices{G_S}$,
then $m\in\sat{G_S}$ by~assumption.
\end{proof}

The next lemma states that the indeterminates contained in a minimal
prime correspond to a disconnector of~$G$, and
Theorem~\ref{t:MinimalPrimesPBEI} below says when the converse is true
as well.

\begin{lemma}\label{l:VariableSetsOfMinimalPrimes}
Let $\pp$ be a minimal prime of~$\pbei{G}$.  There exists a
disconnector $S\subset\vertices{G}$ of $G$ and a minimal prime $\pp'$ of
$\sat{G_S}$ such that $\pp=\mathfrak{m}_S+\pp'$.
\end{lemma}
\begin{proof}
Let $S$ and $\pp'$ be as in Lemma~\ref{l:FormOfMinimalPrimes}.  We
prove that $S$ is a disconnector.  Assume the converse, i.e., there
exists $s\in S$ such that $\{s\}$ is not a disconnector of
$G_{S\setminus\{s\}}$.  According to
Remark~\ref{r:DisconnectorsAndWalkParity} and
Proposition~\ref{p:SubgraphsAndSaturation},
\begin{equation*}
\sat{G_{S\setminus\{s\}}}\subseteq\mathfrak{m}_{\{s\}}+\sat{G_S}\subseteq\mathfrak{m}_{\{s\}}+\pp'.
\end{equation*}
Hence, since the ideal on the right-hand side is prime, choose a
minimal prime $\pp''$ of $\sat{G_{S\setminus\{s\}}}$ such that
$\sat{G_{S\setminus\{s\}}}\subseteq \pp''\subsetneq
\mathfrak{m}_{\{s\}}+\pp'$.  This give rise to
\begin{equation*}
\pbei{G}\subset\mathfrak{m}_{S\setminus\{s\}}+\pp''\subsetneq\mathfrak{m}_{S}+\pp' = \pp
\end{equation*}
which contradicts the minimality of $\pp$. 
\end{proof}

Let $S\subseteq\vertices{G}$ be a disconnector of~$G$.  The induced
subgraph $G_S$ splits into bipartite components
$B_1,\dots,B_{\bComp(G_S)}$ and non-bipartite components
$N_1,\dots,N_{\nbComp(G_S)}$.  By Proposition~\ref{p:decompSat} the
minimal primes of $\sat{G_S}$ are
\begin{equation}\label{eq:satprimes}
\pp=\sum_{i=1}^{\bComp(G_S)}\sat{B_i} +
\sum_{i=1}^{\nbComp(G_S)}\pp^{\sigma_i}(N_i), \text{ where } \begin{cases}
\sigma_i \in \{+,-\}, & \text{ if }\charac(\kk)\neq 2,\\
\sigma_i = +, & \text{ if }\charac(\kk)= 2.
\end{cases}
\end{equation}
Not all of these primes lead to minimal primes of $\pbei{G}$ because
of the following effect.

\begin{example}
\label{ex:ChoicesOfBinomialPartOfSaturation}
Let $G$ be the graph in
Figure~\ref{fig:ChoicesOfBinomialPartOfSaturation}.
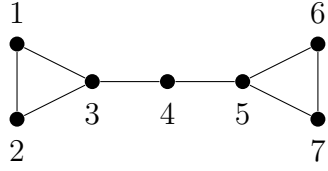
\begin{figure}[bt]
	\begin{tikzpicture}[xscale=0.5,yscale=0.5]
		\node [label={[label distance=1pt]270:$2$},fill, circle, inner sep=2pt](a) at (0,0) {};
		\node [label={[label distance=1pt]90:$1$},fill, circle, inner sep=2pt](b) at (0,2) {};
		\node [label={[label distance=1pt]270:$3$},fill, circle, inner sep=2pt](c) at (2,1) {};
		\node [label={[label distance=1pt]270:$4$},fill, circle, inner sep=2pt](d) at (4,1) {};
		\node [label={[label distance=1pt]270:$5$},fill, circle, inner sep=2pt](e) at (6,1) {};
		\node [label={[label distance=1pt]270:$7$},fill, circle, inner sep=2pt](f) at (8,0) {};
		\node [label={[label distance=1pt]90:$6$},fill, circle, inner sep=2pt](g) at (8,2) {};

		\draw (a)--(b)--(c)--(a);
		\draw (e)--(f)--(g)--(e);
		\draw (c) --(d) --(e);

\end{tikzpicture}
\caption{A graph for which one of the primes in \eqref{eq:satprimes}
is not a minimal prime.}\label{fig:ChoicesOfBinomialPartOfSaturation}
\end{figure}
The vertex $4$ is a disconnector, and $G_{\{4\}}$ consists of the two
triangles $N_1=\{1,2,3\}$ and $N_2=\{5,6,7\}$.  Choosing for both
triangles the positive sign component, we obtain the prime ideal
\begin{equation*}
\mathfrak{m}_{\{4\}}+\nonbipPos{N_1}+\nonbipPos{N_2}=\mathfrak{m}_{\{4\}}+\ideal{x_i+y_i:
i\in[7]\setminus\{4\}} 
\end{equation*} 
which is not minimal over $\pbei{G}$ since it contains the prime ideal
$\nonbipPos{G}$.  On the other hand, both ideals with the binomial
part $\mathfrak{m}_{\{4\}}+\pp^\pm(N_1)+\pp^\mp(N_2)$, each having different
signs on the triangles, are minimal over~$\pbei{G}$.
\end{example}

A combinatorial condition on $\sigma$ in \eqref{eq:satprimes}
guarantees that a minimal prime of $\sat{G_S}$ is the binomial part of
a minimal prime of~$\pbei{G}$ (the monomial part being
$\mathfrak{m}_S$).  To see it, let $s\in S$ be such that
$\cComp(G_S)>\cComp(G_{S\setminus\{s\}})$, i.e., when adding $s$ back
to $G_S$ some of its connected components are joined.  Denote by
$\cC_{G_S}(s)$ the set of only those connected components of $G_S$
which are joined when adding~$s$.

\begin{defn}\label{d:sign-split}
Let $S\subset\vertices{G}$ be a disconnector of~$G$.  A minimal prime
$\pp$ of~$\sat{G_S}$ is \emph{sign-split} if for all $s\in S$ such
that $\cC_{G_S}(s)$ contains no bipartite graphs, the prime summands of
$\pp$ corresponding to connected components in $\cC_{G_S}(s)$ are not
all equal to $\pp^+$ or all equal to~$\pp^-$.
\end{defn}

\begin{remark}
If $\cC_{G_S}(s)$ contains at least one bipartite graph, then
Definition~\ref{d:sign-split} imposes no restriction and every choice
of prime summands is sign-split.
\end{remark}

\begin{remark}\label{r:charNoSplit}
If $\charac(\kk) = 2$, then all signs $\sigma$ in \eqref{eq:satprimes}
are fixed.  In this case, Definition~\ref{d:sign-split} can only be
satisfied if $\cC_{G_S}(s)$ contains a bipartite component for each
$s\in S$.
\end{remark}

\begin{example}\label{e:NotAllDisconnectorsContributeMinP}
Not every disconnector $S\subset\vertices{G}$ of~$G$ admits a
sign-split minimal prime for $\sat{G_S}$, and thus not every
disconnector contributes minimal primes to~$\pbei{G}$.
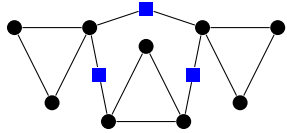
\begin{figure}[htbp]
	\begin{tikzpicture}[xscale=0.5,yscale=0.5]
		\node [fill, circle, inner sep=2pt](a) at (1,2) {};
		\node [fill, circle, inner sep=2pt](b) at (2,0) {};
		\node [fill, circle, inner sep=2pt](c) at (3,2){};

		\node [fill, circle, inner sep=2pt](g) at (4.5,1.5){};
		\node [fill, circle, inner sep=2pt](h) at (3.5,-0.5){};
		\node [fill, circle, inner sep=2pt](i) at (5.5,-0.5){};

		\node [fill, circle, inner sep=2pt](d) at (6,2) {};
		\node [fill, circle, inner sep=2pt](e) at (7,0) {};
		\node [fill, circle, inner sep=2pt](f) at (8,2){};

		\node [fill=blue, rectangle, scale=1.3, inner sep=2pt](j) at (3.25,0.75){};
		\node [fill=blue, rectangle, scale=1.3, inner sep=2pt](k) at (5.75,0.75){};
		\node [fill=blue, rectangle, scale=1.3, inner sep=2pt](l) at (4.5,2.5){};

   \draw (c) -- (j) -- (h);
   \draw (d) -- (k) -- (i);

   \draw (c) -- (l) -- (d);

   \draw (a)--(b)--(c)--(a);
   \draw (d)--(e)--(f)--(d);
   \draw (g)--(h)--(i)--(g);
	\end{tikzpicture}
	\caption{\label{f:NoSignSplit}A disconnector whose binomial
	parts cannot be sign-split.}
\end{figure}
Consider the graph in Figure~\ref{f:NoSignSplit}.  The set of blue
square vertices is a disconnector that does not contribute minimal
primes.  Adding one of the squares back yields the requirement that
the primes on the two now connected triangles have different signs,
but these three requirements cannot be satisfied simultaneously.
\end{example}

\begin{thm}\label{t:MinimalPrimesPBEI}
The minimal primes of $\pbei{G}$ are the ideals
$\mathfrak{m}_S+\pp$, where $S\subseteq\vertices{G}$ is a
disconnector of $G$ and $\pp$ is a sign-split minimal prime
of~$\sat{G_S}$.
\end{thm}
\begin{proof}
According to Lemma~\ref{l:VariableSetsOfMinimalPrimes}, all minimal
primes of $\pbei{G}$ have the form~$\mathfrak{m}_S+\pp$, where
$S\subset\vertices{G}$ is a disconnector and $\pp$ is a minimal prime
of~$\sat{G_S}$.  We first show that if $\pp$ is sign-split, this ideal
is minimal over~$\pbei{G}$.  Assume not, then by
Lemma~\ref{l:FormOfMinimalPrimes} there exists a set
$T\subseteq\vertices{G}$ and a minimal prime $\tilde{\pp}$ of
$\sat{G_T}$ such that
\begin{equation}\label{equ:InclusionMinimalPrimes}
\pbei{G}\subseteq\mathfrak{m}_T+\tilde{\pp}\subsetneq
\mathfrak{m}_S+\pp.  
\end{equation} 
This implies $T\subsetneq S$, since if $T=S$, then by
Lemma~\ref{l:FormOfMinimalPrimes} also $\tilde{\pp}=\pp$.  Let
$s'\in S\setminus T$, then
$G_S\subsetneq G_{S\setminus\{s'\}}\subseteq G_T$.  Since $s'$ is a
disconnector of $G_{S\setminus\{s'\}}$,
$\sComp(G_{S})>\sComp(G_{S\setminus\{s'\}})$.  Let again
$\cC_{G_S}(s')$ be the set of connected components in $G_S$ that are
joined to $s'$ in $G_{S\setminus\{s'\}}$.
If $\cC_{G_S}(s')$ contains at least one bipartite component, adding
$s'$ to $G_S$ either this component becomes non-bipartite in
$G_{S\setminus\{s'\}}$ or it is joined to another bipartite component
of~$G_S$.  In the first case, let $B$ be a bipartite component which
becomes non-bipartite.  There exists $i\in\vertices{B}$ such that
$x_i^2-y_i^2\in\sat{G_{S\setminus\{s'\}}}\subset\sat{G_T}\subset\tilde{\pp}$,
but $x_i^2-y_i^2\not\in\sat{B}$.  Since $\sat{B}$ is a summand of
$\pp$, $x_i^2-y_i^2\not\in\mathfrak{m}_S+\pp$, in
contradiction to \eqref{equ:InclusionMinimalPrimes}. In the second
case, let $B_1$ and $B_2$ be the bipartite components of $G_S$ which
are joined to~$s'$.  There are $i_1\in\vertices{B_1}$ and
$i_2\in\vertices{B_2}$ such that there exists an $(i_1,i_2)$-walk
in~$G_{S\setminus\{s'\}}$.  Independent of the parity of this walk,
the corresponding Markov move is not contained in
$\sat{B_1}+\sat{B_2}$ since there is no applicable move from the
Graver basis. Since $\sat{B_1}$ and $\sat{B_2}$ are summands of $\pp$
involving the indeterminates $i_1$ and $i_2$, there is a binomial which is
not in $\mathfrak{m}_S+\pp$ but in
$\sat{G_{S\setminus\{s'\}}}\subset\tilde{\pp}$
contradicting~\eqref{equ:InclusionMinimalPrimes}.

Assume now that all components in $\cC_{G_S}(s')$ are non-bipartite
(there are at least two of them since $s'$ is a disconnector).  By
assumption, $\pp$ is sign-split, i.e., there exist distinct
components $N_1,N_2\in\cC_{G_S}(s)$ such that $\nonbipPos{N_1}$ and
$\nonbipNeg{N_2}$ are summands of~$\pp$.  There is an odd walk from a
vertex $i_1\in V(N_1)$ to a vertex $i_2\in V(N_2)$ in
$G_{S\setminus\{s'\}}$, and therefore,
$x_{i_1}x_{i_2}-y_{i_1}y_{i_2}\in\sat{G_{S\setminus\{s'\}}}\subset\tilde{\pp}$.
Since
\begin{equation*}
x_{i_1}x_{i_2}-y_{i_1}y_{i_2}\not\in\nonbipPos{N_1}+\nonbipNeg{N_2},
\end{equation*}
also $x_{i_1}x_{i_2}-y_{i_1}y_{i_2}\not\in\pp$.  By construction,
$i_1,i_2\not\in S$ and thus
\begin{equation*}
	x_{i_1}x_{i_2}-y_{i_1}y_{i_2}\not\in\mathfrak{m}_S+\pp
\end{equation*}
which contradicts~\eqref{equ:InclusionMinimalPrimes}. This shows
minimality of $\mathfrak{m}_S+\pp$.

Let now $\mathfrak{m}_S+\pp$ be a minimal prime of~$\pbei{G}$.
The set $S$ is a disconnector by
Lemma~\ref{l:VariableSetsOfMinimalPrimes} and thus it remains to prove
that $\pp$ is sign-split.  To the contrary, assume there is a vertex
$s\in S$ with $\cComp(G_{S\setminus\{s\}})>\cComp(G_S)$ such that
$\cC_{G_S}(s) = \{N_1,\ldots,N_k\}$ consists exclusively of
non-bipartite components, $k \ge 2$, and all summands of $\pp$
corresponding to $N_i$ have the same sign, say~$+$.  When adding $s$
back to $G_{S}$, the components in $\cC_{G_S}(s)$ are joined to a
single, non-bipartite connected component $H$ in
$G_{S\setminus\{s\}}$, whereas all other components of $G_S$ coincide
with connected components of~$G_{S\setminus\{s\}}$.  Since
\begin{equation*}
\nonbipPos{H}=\mathfrak{m}_{\{s\}}+\sum_{i=1}^k\nonbipPos{N_i}\subsetneq\mathfrak{m}_{\{s\}}+\sum_{i=1}^k\nonbipPos{N_i},
\end{equation*}
choosing on all other components of $G_{S\setminus\{s\}}$ the same
prime component as in $G_S$, we obtain a prime ideal that is strictly
smaller than~$\mathfrak{m}_S+\pp$.
\end{proof}

\begin{remark}
Example~\ref{ex:ChoicesOfBinomialPartOfSaturation} and
Definition~\ref{d:sign-split} are valid independent of $\charac(\kk)$.  In
the above proof, the case of $\charac(\kk) = 2$ could be simplified,
but everything works in general without the need for a case
distinction.
\end{remark}

\section{Radicality and mesoprimary decomposition}
\label{sec:Radicality}

The intersection of the minimal primes of $\pbei{G}$ depends on
$\charac(\kk)$ so that we do not attempt to compute it directly.
Theorem~\ref{t:radical} below says that $\pbei{G}$ is radical if the
characteristic is not two.  Here is the principal source of field
dependence (see also~\cite[Theorem~1.2]{Herzog2010}).

\begin{remark}\label{r:RadicalityOverFiniteField}
Fix a field $\kk$ with $\charac(\kk)=2$.  The parity binomial edge
ideal $\pbei{G}$ is radical in $\kk[\xx,\yy]$ if and only if $G$ is
bipartite. Clearly, if $G$ is bipartite, then $\pbei{G}$ is radical by
Remark~\ref{r:bipartite-bei}.  Conversely, let $(i_1,\ldots,i_{r+1})$
with $i_{r+1}=i_1$ be an odd cycle in $G$. According to
Lemma~\ref{l:Walks}, $((x_{i_1}-y_{i_1})y_{i_2}\cdots
y_{i_{r}})^2=(x_{i_1}^2-y_{i_1}^2)y_{i_2}^2\cdots
y_{i_{r}}^2\in\pbei{G}$.  The only possible reduced binomials whose
leading monomials divide $x_{i_1}y_{i_2}\cdots y_{i_{r}}$ correspond
to minimal even $(i_1,i_k)$-walks in $G[\{i_1,\ldots,i_{r}\}]$ with
$k\in\{2,\ldots,r\}$. Replacements coming from these binomials lead to
monomials where $x_{i_1}$ is replaced by $y_{i_1}$ and $y_{i_k}$ is
replaced by $x_{i_k}$. Thus, $x_{i_1}y_{i_2}\cdots
y_{i_{r}}\not\equiv_{\pbei{G}}y_{i_1}y_{i_2}\cdots y_{i_{r}}$ and
hence $\pbei{G}$ is not radical.
\end{remark}

\begin{remark}\label{r:multihomog}
The ideal $\pbei{G}$ is homogeneous with respect to the multigrading
$\deg(x_i)=\deg(y_i)=e_i$, where $e_i$ is the $i$-th standard basis
vector of~$\RR^{|\vertices{G}|}$.
\end{remark}

\begin{lemma}\label{l:bothMons}
Let $i\in \vertices{G}$ and $m\in\pbei{G} + \mathfrak{m}_{\{i\}}$ be a
monomial.  Then $m\in\mathfrak{m}_{\{i\}}$.
\end{lemma}
\begin{proof}
Since it is generated by pure difference binomials, $\pbei{G}$ does
not contain any monomials.  Thus, any monomial in
$\pbei{G} + \mathfrak{m}_{\{i\}}$ is equivalent to one in
$\mathfrak{m}_{\{i\}}$ modulo term replacements using binomials in
$\pbei{G}$, but these do not change membership
in~$\mathfrak{m}_{\{i\}}$ by Remark~\ref{r:multihomog}.
\end{proof}

\begin{prop}\label{p:Intersection}
For any graph $G$,
$\pbei{G}=\sat{G}\cap\bigcap_{i\in\vertices{G}}(\pbei{G}+\mathfrak{m}_{\{i\}}).$
\end{prop}
\begin{proof}
According to \cite[Corollary 1.5]{es96}, the intersection is binomial.
Let $b$ be any binomial in the intersection.  For each
$i\in\vertices{G}$, there are three cases: Either no term of $b$ is
individually contained in $\pbei{G} + \mathfrak{m}_{\{i\}}$, exactly
one is, or both are.  In the first case, \cite[Proposition~1.10]{es96}
implies $b \in \pbei{G}$.  In the second case, it implies that the
other monomial is contained in $\pbei{G}$, which is impossible.  Thus
it suffices to consider binomials $b$ both of whose monomials are
contained in $\pbei{G} + \mathfrak{m}_{\{i\}}$ for all
$i\in\vertices{G}$.  By Lemma~\ref{l:bothMons}, both monomials of $b$
are contained in $\mathfrak{m}_{\{i\}}$ for each $i\in\vertices{G}$.
Since $b\in\sat{G}$, there exist Markov moves
$m_{s_1t_1},\ldots,m_{s_rt_r}$ corresponding to
$(s_1,t_1),\ldots,(s_r,t_r)$-walks, respectively, such that
\[
b=x^{h_1}y^{h'_1}m_{s_1t_1}+\cdots+x^{h_r}y^{h'_r}m_{s_rt_r}
\]
with $h_i,h_i'\in\NN^n$.  We can assume that one monomial of $b$
equals one of the monomials of $x^{h_1}y^{h_1}m_{s_1t_1}$.  Thus both
monomials of $x^{h_1}y^{h_1}m_{s_1t_1}$ are divisible by at least one
indeterminate for each $i\in\vertices{G}$ and, by Lemma~\ref{l:Walks},
$x^{h_1}y^{h_1}m_{s_1t_1} \in \pbei{G}$.  Replacing $b$ by
$b-x^{h_1}y^{h_1}m_{s_1t_1}$ and iterating the argument eventually
yields $b\in\pbei{G}$.
\end{proof}

\begin{thm}\label{t:radical}
Let $G$ be a graph. If $\textnormal{char}(\kk)\neq 2$, then $\pbei{G}$
is a radical ideal.
\end{thm}
\begin{proof}
The proof is by induction on the number of vertices of~$G$.  If $G$
has at most one vertex, then $\pbei{G}=0$ and the claim holds.
Proposition~\ref{p:decompSat} shows that
$\pbei{G}+\mathfrak{m}_{\{i\}}=\pbei{G_{\{i\}}}+\mathfrak{m}_{\{i\}}$ for all
$i\in\vertices{G}$.  Thus Theorem~\ref{p:Intersection} reads as
$\pbei{G}=\sat{G}\cap\bigcap_{i=1}^n(\pbei{G_{\{i\}}}+\mathfrak{m}_{\{i\}})$.
By the induction hypothesis, $\pbei{G_{\{i\}}}$ is radical and thus
$\pbei{G_{\{i\}}}+\mathfrak{m}_{\{i\}}$ is radical.
Proposition~\ref{p:decompSat} says that $\sat{G}$ is radical if
$\charac(\kk) \neq 2$ which yields the result.
\end{proof}

Theorem~\ref{t:charzwo} below contains a primary decomposition of
$\pbei{G}$ in the case $\charac(\kk) = 2$.  It uses the following
lemma, which allows to transport decompositions between different
characteristics.  Recall that the combinatorics of any binomial ideal
$I$ is encoded in its congruence $\til_I$ which identifies monomials
$m_1,m_2$, whenever $m_1-\lambda m_2 \in I$ for some
nonzero~$\lambda\in\kk$.  A binomial ideal is \emph{unital} if it is
generated by monomials and pure differences of monomials.  Then each
congruence is the congruence of a unital binomial ideal, though not
uniquely.
\begin{lemma}\label{l:unitalIntersect}
If a decomposition $I = J_1 \cap \dots \cap J_s$ of a unital binomial
ideal $I$ into unital binomial ideals $J_i$, $i=1,\dots,s$ is valid in
some characteristic, then it is valid in any characteristic.
\end{lemma}
\begin{proof}
The congruence $\til_I$ induced by $I$ is the common refinement of the
congruences $\til_{J_i}$, induced by the $J_i$, $i=1,\dots, s$.  Thus,
in any characteristic, \cite[Theorem~9.12]{kahle11mesoprimary} implies
that $I$ and $J_1\cap \dots \cap J_s$ can only differ if one of them
contains monomials, but the other does not.  This cannot happen since
unital binomial ideals contain monomials if and only if they have
monomials among the generators.
\end{proof}

According to Example~\ref{e:NotAllDisconnectorsContributeMinP}, not
all disconnectors contribute minimal primes.  From
Definition~\ref{d:sign-split} it may seem that this is an arithmetic
effect.  It is not; the primary decomposition of $\pbei{G}$ in
characteristic two also witnesses it.  For the following definition,
recall that a hypergraph is \emph{$k$-colorable} if the vertices can be
colored with $k$ colors so that no edge is monochromatic.
\begin{defn}\label{d:effective}
Let $S \subset\vertices{G}$ be a disconnector, and let
$s_1,\ldots,s_r\in S$ be the vertices such that $\cC_{G_S}(s_i)$
consists exclusively of non-bipartite components of~$G_S$.  Let
$\mathcal{H}$ be the hypergraph whose vertex set consists of
the connected components $\cC_{G_S}(s_1)\cup\ldots\cup\cC_{G_S}(s_r)$
and with edge set $\{\cC_{G_S}(s_1),\ldots,\cC_{G_S}(s_r)\}$.  The
disconnector~$S$ is \emph{effective} if $\mathcal{H}$ is $2$-colorable.
\end{defn}

\begin{remark}
A disconnector is effective if and only if, in characteristic zero, it
admits sign-split minimal primes.
\end{remark}

\begin{thm}\label{t:charzwo}
Let $\mathcal{S}$ be the set of effective disconnectors of~$G$. Then
\begin{equation}\label{eq:mesoOrNot}
\pbei{G} = \bigcap_{S\in\mathcal{S}}\left(\mathfrak{m}_S +
  \sat{G_S}\right).
\end{equation}
If $\charac(\kk)=2$, then \eqref{eq:mesoOrNot} is a primary
decomposition of~$\pbei{G}$.
\end{thm}
\begin{proof}
For each disconnector $S\in\mathcal{S}$, let $B^S_1,\dots,B^S_{\bComp(G_S)}$ be the
bipartite components and $N^S_1,\dots,N^S_{\nbComp(G_S)}$ the
non-bipartite components of~$G_S$.  Let
$\Sigma^S \subset \{+,-\}^{\nbComp(G_S)}$ denote the set of sign
patterns that are sign-split.  In characteristic zero, by
Theorems~\ref{t:MinimalPrimesPBEI} and~\ref{t:radical}, $\pbei{G}$
decomposes as
\[
\pbei{G} = \bigcap_{S\in\mathcal{S}} \bigcap_{\sigma \in \Sigma^S}
\left(\mathfrak{m}_S + \sum_{i=1}^{\bComp(G_S)} \sat{B^S_i} +
  \sum_{i=1}^{\nbComp(G_S)} \pp^{\sigma_i}(N^S_i) \right).
\]
The intersection remains valid when intersecting over additional
ideals containing~$\pbei{G}$.  In particular, the sign-split
requirement can be dropped and $\Sigma^S$ replaced
by~$\{+,-\}^{\nbComp(G_S)}$.  Carrying out this inner intersection
yields the ideals $\mathfrak{m}_S+\sat{G_S}$ by
Proposition~\ref{p:decompSat}, and hence \eqref{eq:mesoOrNot} is valid
in characteristic zero.  Since all involved ideals are unital,
Lemma~\ref{l:unitalIntersect} yields that \eqref{eq:mesoOrNot} is
valid in any characteristic.  The ideals under consideration are
primary if $\charac(\kk)=2$ according to Proposition~\ref{p:decompSat}
and thus the second statement follows.
\end{proof}

The technique of adding ``phantom components'' to a primary
decomposition so that it faithfully exists over some other field (as
we did in the proof of Theorem~\ref{t:charzwo}) was mentioned in the
introduction of~\cite{kahle11mesoprimary} as one way to arrive at more
combinatorially accurate decompositions of binomial ideals.  The
upshot of the rest of the paper is that the primary decomposition in
characteristic two is an example of a mesoprimary decomposition.  This
means not only that all ideals in \eqref{eq:mesoOrNot} are
mesoprimary, but additionally each of the intersectands witnesses a
combinatorial feature of the graph that the binomials of~$\pbei{G}$
induce on the monomials of~$\kk[\xx,\yy]$.  Generally it can be quite
challenging to determine a mesoprimary decomposition, exactly because
of the stringent combinatorial conditions that it has to meet.  Here
it is mostly a translation of the (involved) definitions, essentially
because all ideals are unital and the ambient ring is the polynomial
ring~\cite[Remark~12.8]{kahle11mesoprimary}.  We refrain from
introducing too much of the machinery from~\cite{kahle11mesoprimary}
here, but do employ their notation.  In the following, we give explicit
references to all relevant definitions.

\begin{thm}\label{t:meso}
The decomposition \eqref{eq:mesoOrNot} is a mesoprimary decomposition
of~$\pbei{G}$.
\end{thm}
\begin{proof}
For $S\in\mathcal{S}$, let $J_S:=\mathfrak{m}_S+\sat{G_S}$ be the
intersectand corresponding to~$S$.  The ideal $J_S$ is
$P_S$-mesoprimary where $P_S \subset \NN^{2|\vertices{G}|}$ is the
monoid prime ideal~$\<e_{x_i},e_{y_i} : i \in S\>$.  In fact (like any
ideal that equals a lattice ideal plus monomials in a disjoint set of
indeterminates), $J_S$ is mesoprime
\cite[Definition~10.4]{kahle11mesoprimary} since it equals the kernel
of the monomial homomorphism
\[
\kk[\xx,\yy] \to \kk[x_i^\pm,y_i^\pm : i\notin S]/\sat{G_S} =
\kk[\ZZ^{2|\vertices{G} \setminus S|}/L_s]
\] which maps $x_i,y_i$ to zero if $i\in S$ and to their images in the
Laurent ring if $i\notin S$.  Here $L_S$ is the image in
$\ZZ^{2|\vertices{G} \setminus S|}$ of the adjacency matrix of~$G_S$.

According to \cite[Definition~13.1]{kahle11mesoprimary}, it remains to
show that at each cogenerator \cite[Definitions~7.1
and~12.16]{kahle11mesoprimary} of $J_S$, the $P_S$-mesoprimes of
$\pbei{G}$ and $J_S$ agree.  Let~$J_S^\pm$ be the image of $J_S$ in
$R^S = \kk[x_i,y_i, i\in S, x^\pm_j,y^\pm_j, j\notin S]$.  The
cogenerators of $J_S$ are monomials in $\kk[\xx,\yy]$ whose images in
$R_S/J_S^\pm$ are annihilated by
$\mathfrak{m}_S$.  Since~$J_S$ contains
$\mathfrak{m}_S$, the cogenerators are simply all monomials in the
indeterminates $x_i,y_i$ for $i\notin S$.  Now the $P_S$-mesoprime of the
mesoprime $J_S$ (at any monomial) is just~$J_S$.  Thus it remains to
compute the $P_S$-mesoprime of $\pbei{G}$ at any cogenerator.
Translating \cite[Definition~11.11]{kahle11mesoprimary} to
$\kk[\xx,\yy]$, this mesoprime is given by
$\left(\pbei{G} + \mathfrak{m}_S\right) : \left(\prod_{i\notin S}
x_iy_i\right)^\infty$.  The result now follows by Lemma~\ref{l:bothMons}.
\end{proof}

The stringent combinatorial conditions that guarantee a canonical
mesoprimary decomposition require additional knowledge about the
witness structure of~$\pbei{G}$.

\begin{conj}\label{c:better}
The mesoprimary decomposition in \eqref{eq:mesoOrNot} is combinatorial
and characteristic.
\end{conj}

To prove Conjecture~\ref{c:better} one needs precise control over the
various witnesses that contribute to coprincipal
decompositions~\cite[Theorems~8.4 and~16.9]{kahle11mesoprimary}.
Experiments with \textsc{Macaulay2} indicate that the mesoprimary
decomposition of the congruence $\til_{\pbei{G}}$ differs
significantly from that of the ideal~$\pbei{G}$.  For example, if $G$
is a path $G = 1-2-3-4-5$, then $\pbei{G}$ has the following
mesoprimary decomposition:
\begin{align*}
  \pbei{G} = \sat{G} & \cap (\mathfrak{m}_{\{4\}} + \sat{1-2-3}) \cap
  (\mathfrak{m}_{\{2\}} + \sat{3-4-5}) \\ 
		     & \cap (\mathfrak{m}_{\{3\}} + \pbei{1-2} + \pbei{4-5}) \cap
           \mathfrak{m}_{\{2,4\}}.
\end{align*}
Intersecting all but the last ideal yields the ideal
\[
\pbei{G} + \ideal{x_1x_3y_3y_5 - x_1x_5y_3^2 - x_3^2 y_1y_5 +
x_3x_5y_1y_3}.
\]
This ideal has the same binomials as $\pbei{G}$ and thus induces the
same congruence.  The monomial ideal that was omitted does not
influence the congruence.  Its sole purpose is to cut away
non-binomials.  The monoid prime
$\<e_{x_2},e_{y_2},e_{x_4},e_{y_4}\> \subset \NN^{2|\vertices{G}|}$
contributes only non-key witnesses to the case.  Nevertheless these
witnesses are essential in the sense
of~\cite[Definition~12.1]{kahle11mesoprimary}.

\bibliographystyle{amsplain}
\bibliography{pbei}

\end{document}